\documentclass[11pt]{amsart}

\usepackage{amssymb}
\usepackage[T1]{fontenc} 
\usepackage{wasysym} 
\usepackage{url}
\usepackage{relsize} 

\usepackage{color}

\usepackage[all,2cell,ps,cmtip]{xy}
\newcommand{\piccirc}{\protect{\text{\fontsize{14pt}{12pt}\selectfont$\circ$}}}

\theoremstyle{plain}
\newtheorem{theorem}{Theorem}[section]
\newtheorem{proposition}[theorem]{Proposition}
\newtheorem{lemma}[theorem]{Lemma}
\newtheorem{corollary}[theorem]{Corollary}
\newtheorem{fact}[theorem]{Fact}

\newtheorem*{B-E-theorem}{Blok-Esakia Theorem}
\newtheorem*{B-E-theorem-ex}{Blok-Esakia Theorem extended version}
\newtheorem*{B-E-theorem-res}{Blok-Esakia Theorem restricted versions}
\newtheorem*{B-lemma}{Blok Lemma}
\newtheorem*{B-lemma-fc}{Blok Lemma: finite case}

\theoremstyle{definition}

\newcommand{\0}{\emptyset}
\renewcommand{\=}{\approx}
\renewcommand{\leq}{\leqslant}
\renewcommand{\geq}{\geqslant}
\newcommand{\meet}{\wedge}
\newcommand{\bigmeet}{\bigwedge}
\newcommand{\join}{\vee}
\newcommand{\bigjoin}{\bigvee}

\newcommand{\mmeet}{\wedge\mbox{}\kern-10pt\wedge}
\newcommand{\mbigmeet}{\bigwedge\mbox{}\kern-10pt\bigwedge}
\newcommand{\mjoin}{\vee\mbox{}\kern-10pt\vee}
\newcommand{\mbigjoin}{\bigvee\mbox{}\kern-10pt\bigvee}
\newcommand{\mneg}{\text{\raisebox{-.5pt}{$\neg$}\mbox{}\kern-6.7pt\raisebox{1pt}{$\neg$}}}

\newcommand{\mBox}{\Box}

\newcommand{\Mod}{\operatorname{Mod}}

\newcommand{\Thel}{{ \sf Th^{el}}}
\newcommand{\diagplus}{\sf{diag^+}}

\newcommand{\Hey}{\mathcal H}
\newcommand{\Grz}{\mathcal G}

\newcommand{\V}{\mathcal V}
\newcommand{\W}{\mathcal W}
\newcommand{\U}{\mathcal U}
\newcommand{\K}{\mathcal K}
\newcommand{\Q}{\mathcal Q}

\newcommand{\M}{\mathbf{M}}
\newcommand{\N}{\mathbf{N}}
\newcommand{\A}{\mathbf{A}}
\newcommand{\B}{\mathbf{B}}
\newcommand{\C}{\mathbf C}

\renewcommand{\S}{\mathbf S}
\newcommand{\Fr}{\mathbf{F}}

\renewcommand{\H}{\mathbf{H}}
\newcommand{\Nat}{\mathbb N}

\renewcommand{\phi}{\varphi}

\begin{document}

\title{On the Blok-Esakia theorem for universal classes}

\author{Micha\l~M. Stronkowski}

\address{Faculty of Mathematics and Information Sciences,
Warsaw University of Technology, ul. Koszykowa 75, 00-662
Warsaw, Poland}
\email{m.stronkowski@mini.pw.edu.pl}



\thanks{The work was supported by the Polish National Science Centre grant no. DEC- 2011/01/D/ST1/06136.}

\begin{abstract}
The Blok-Esakia theorem states that there is an isomorphism from the lattice of intermediate logics onto the lattice of normal extensions of Grzegorczyk modal logic. The extension for multi-conclusion consequence relations was obtained by Emil Je\v r\'abek as an application of canonical rules. We show that Je\v r\'abek's result follows already from Blok's algebraic proof. 


We also prove that the properties of strong structural completeness and strong universal completeness are preserved and reflects by the aforementioned isomorphism. These properties coincide with structural completeness and universal completeness respectively for single-conclusion consequence relations and, in particular, for logics. 
\end{abstract}

\maketitle

\section{Introduction}
Let us consider the modal formula
\[ 
grz=\mBox(\mBox(p\to \mBox p)\to p)\to p.
\]
Let GRZ be the normal extension of S4 modal logic axiomatized by $grz$ \cite{Grz67}. Let INT be the intuitionistic logic.

The celebrated Blok-Esakia theorem states that there is an isomorphism from the lattice of extensions of INT onto the lattice of normal extensions of GRZ \cite{Blo76,Esa76}. Algebraically, it means that there is an isomorphism from the lattice of varieties of Heyting algebras onto the lattice of varieties of modal Grzegorczyk algebras. (The reader may find a more detailed history of {\it classical} interpretations of intuitionistic logic in  papers \cite{CZ92,Mur06,WZ14}.)

Nowadays we have various proofs of this result. The one of Blok is purely algebraic. Esakia used Esakia spaces, i.e., certain topological relational structures. There is also akharyashchev's proof based on his canonical formulas \cite{CZ97,Zak84}.

There are many extensions of the Blok-Esakia theorem, see a recent survey \cite{WZ14} and the references therein. Let us recall one of them. In \cite{Jer09} Je\v r\'abek extended the Blok-Esakia theorem to multi-conclusion consequence relations (called there rule systems). It means that he allows not only axiomatic extensions, but also extensions obtained by adding new inference rules (with possibly many conclusions). Algebraically Je\v r\'abek's result says that there is an isomorphism from the lattice of universal classes of Heyting algebras onto  the lattice of universal classes of modal Grzegorczyk algebras. Je\v r\'abek's proof is based on canonical rules, a tool which extents Zakharyashchev's canonical formulas.
The starting point of this project was an observation that actually Je\v r\'abek's result follows directly from Blok's algebraic proof.

There is a quite developed algebraic theory for single-conclusion consequence relations, see e.g. \cite{Fon16,FJP03,Cze01}. However multi-conclusion consequence relations, despite being known for decades \cite{SS08},  seems to be neglected. They appear naturally in proof theory in the context of sequent calculi \cite{NP01}. But there they are used mainly as a handy tool for describing logics and are not object of studies per se. The situation has been changed with Je\v r\'abek's paper \cite{Jer09} and his observation that multi-conclusion inference rules may be used for the canonical axiomatization of intermediate and modal logics. This topic was recently undertaken in many papers \cite{BBI16,BBIl16,BBIl16a,BGGJ16,BG14,Cit15,Cit15a,Cit16,Cit16a,Gou15,Iem15,Iem14b,Jer15}.  
\vspace{.7em}

The paper contains a full presentation of an algebraic proof of Je\v r\'abek's result. We simplify one point  of Blok's reasoning. Namely, we provide a short proof that free Boolean extensions of Heyting algebras are Grzegorczyk algebras (this is mainly done in Section \ref{sec:: Grzegorczyk algebras}). Along the line, we also obtain a new proof of Blok's  characterization of Grzegorczyk algebras.
Still, the main technical ingredient is the Blok lemma which connects free Boolean extensions with Grzegorczyk algebras (its full proof is given in the appendix). We failed to find any essential simplification of its tricky proof.

We also undertake the problem of preservation and  reflection  of some properties by the Blok-Esakia isomorphisms. For logics the topic was thoroughly investigated in the past and summarized in the survey \cite{CZ92}. 
 We verify  preservation and  reflection of strong structural completeness and strong universal completeness. For single-conclusion consequence relations and logics they are equivalent to structural completeness and universal completeness respectively. For logics this fact was proved by Rybakov \cite{Ryb97}. (We used it in \cite{DS16} in order to construct a normal extension of S4 modal logic which is almost structurally complete but does not have projective unification.) 


\section{Multi-conclusion consequence relations}

Let us fix a language $\mathcal L$ (i.e., an infinite set of variables and a set of symbols of operations with ascribed arrives). Let $\bf Form$ be the algebra of all formulas (terms) in $\mathcal L$. An \emph{inference rule} (in $\mathcal L$) is an ordered pair, written as $\Gamma/\Delta$, of finite subsets of $Form$.  A set of inference rules, written as a relation $\vdash$, is a \emph{multi-conclusion consequence relation}, \emph{mcr} in short (in \cite{BG14,Jer09} it is called a \emph{rule system}) if  for every finite subsets $\Gamma,\Gamma',\Delta,\Delta'$ of $Form$, for every $\varphi\in Form$ and for every substitution (i.e., an automorphism of $\bf Form$) $\sigma$ the following conditions are satisfied
\begin{itemize}
\item $\{\varphi\}\vdash\{\varphi\}$;
\item if $\Gamma\vdash\Delta$, then $\Gamma\cup\Gamma'\vdash\Delta\cup\Delta'$;
\item if $\Gamma\vdash\Delta\cup\{\varphi\}$ and $\Gamma\cup\{\varphi\}\vdash\Delta$, then $\Gamma\vdash\Delta$;
\item if $\Gamma\vdash\Delta$, then $\sigma(\Gamma)\vdash\sigma(\Delta)$.
\end{itemize}

In this paper we are interested only in intermediate and modal mcrs. Let $K$ denotes modal logic $K$ and $INT$ denotes the intuitionistic logic, both interpreted as  set of formulas. Then an \emph{intermediate mcr} is an mcr $\vdash$ in the language of $INT$ such that $\0\vdash\{\varphi\}$ for every $\varphi\in INT$ and $\{p,p\to,q\}\vdash q$. And a \emph{modal mcr} is an mcr $\vdash$ in the language of $K$ such that $\0\vdash\{\varphi\}$ for every $\varphi\in K$, $\{p,p\to,q\}\vdash q$ and $\{p\}\vdash\{\mBox p\}$. A (general, intermediate or modal) mcr is \emph{axiomatized} by a set $R$ of inference rules if it is a least (general, intermediate or modal receptively) mcr containing $R$.

Intermediate and modal mcrs have algebraic semantics. Let us recall that a \emph{Heyting algebra} is a bounded lattice endowed with the binary operation $\to$ such that $a\meet b\leq c$ iff $a\leq b\to c$ for every triple $a,b,c$ of its elements. It appears that the class of all Heyting algebras form a variety which constitutes a semantics for $INT$. A \emph{modal algebra} is a Boolean algebra endowed with additional unary operation $\mBox$  such that for all its elements $a,b$ we have $\mBox(a\meet b)=\mBox a\meet \mBox b$ and $\mBox 1=1$.
The variety of modal algebras gives a semantics for $K$.

We say that a class $\U$ of algebras is universal iff it is axiomatizable by first order sentences of the form
\begin{equation*}
(\forall \bar x) [s_1\=s'_1\,\sqcap\cdots\,\sqcap s_m\=s'_m\,\Rightarrow\,t_1\=t'_1\,\sqcup\cdots\,\sqcup t_n\=t'_n],
\end{equation*}
where $n$ and $m$ are natural numbers not both equal to zero and  $s_i,s_i',t_j,t_j'$ are arbitrary terms
(We use the symbol $\sqcap$ for first order conjunction, $\sqcup$ for first order disjunction and $\Rightarrow$ for first order implication. The symbols $\meet$, $\join$ and $\to$ will denote operations in algebras.)
We call such formulas \emph{disjunctive universal sentences}. When $m=0$ we talk about \emph{disjunctive universal positive sentences}, when $n=1$ about \emph{quasi-identities}, when $m=0$ and $n=1$ about \emph{identities}. Recall also that a class of algebras is a \emph{universal positive class} if it is axiomatizable by disjunctive universal positive sentences, a \emph{quasivariety} if it is axiomatizable by quasi-identities, and a \emph{variety} if it is axiomatizable by identities.

Let $r$ be an inference rule $\{\varphi_1,\ldots,\varphi_m\}/\{\psi_1,\ldots,\psi_n\}$. By the \emph{translation} of $r$ we mean the disjunctive universal sentence ${\sf T}(r)$ given by
\[
(\forall \bar x)[\varphi_1\= 1\sqcap\cdots\sqcap\varphi_m\= 1\,\Rightarrow\, \psi_1\= 1 \,\sqcup\cdots\,\sqcup \psi_n\= 1].
\]

An inference rule $\Gamma/\Delta$ is a \emph{single-conclusion inference rule} if $|\Delta|=1$, a multi-theorem if $\Gamma=\0$, and a theorem if $\Gamma=\0$ and $|\Delta|=1$. The translation of a single-conclusion inference rule is a quasi-identity, of a multi-theorem is a disjunctive universal positive sentence, and of theorem is an identity.

The following completeness theorem follows from \cite[theorem 2.2]{Jer09}, see also \cite[Theorem 2.5 in Appendix]{BG14} for the modal case.

\begin{theorem}\label{thm:: completeness}
Let $\vdash$ be an intermediate or modal mcr axiomatized by a set of inference rules $R$. Let $\U_{\,\vdash}$ be the universal class of Heyting or modal algebras respectively axiomatized by ${\sf T}(R)$. Then for every inference rule $r$ we have  $r\in\, \vdash$ if and only if $\U_{\,\vdash}\models{\sf T}(r)$.
\end{theorem}



Note that the choice of the axiomatizing set $R$ of inference rules in Theorem \ref{thm:: completeness} 
is arbitrary. Indeed, whichever axiomatizing set for $\vdash$ we choose, we obtain the same class $\U_{\,\vdash}$ 
Moreover, the assignment $\vdash\;\mapsto\; \U_{\,\vdash}$ 
is injective. In addition, every disjunctive universal sentence in the language of Heyting or modal algebras is equivalent to some $\sf T(r)$ in the class of all Heyting or modal algebras. Thus the assignment $\vdash\;\mapsto\; \U_{\,\vdash}$ 
is also surjective.   
Hence the above facts allows us to switch completely in further considerations from mcrs 
to universal classes 
We say that $\vdash$ and $\U_{\,\vdash}$ 
 \emph{correspond to} each other.

\section{Grzegorczyk algebras}
\label{sec:: Grzegorczyk algebras}

A modal algebra $\M$ is called an $\emph{interior algebra}$ if for every  $a\in M$ it satisfies
\[
\mBox\mBox a = \mBox a\leq a.
\]
An element $a$ of $\M$ is \emph{open $\M$} if $\mBox a=a$. Recall that for a modal algebra there is one to one correspondence between its congruences and its open filters, i.e., Boolean filters closed under $\mBox$ operation. It is given by $\theta\mapsto 1/\theta$ and $F\mapsto \theta_F$, where $(a,b)\in \theta_F$ iff $a\leftrightarrow b\in F$. (Here $\theta$ is a congruence and $F$ is an open filter.) In particular, for an interior algebra, an element $b$ belongs to the open filter generated by $a$ iff $\mBox a\leq b$. It follows that an interior algebra is subdirectly irreducible iff it has a largest non-top open element. And an interior algebra is simple iff it has exactly two open elements 0 and 1.

An interior algebra $\M$ is a \emph{Grzegorczyk algebra} if it also satisfies
\begin{equation}\tag{Grz}
\mBox(\mBox(a\to\mBox a)\to a)\leq a
\end{equation}
for every $a\in M$ \cite{Esa79a}.

Recall that the variety of interior/Grzegorczyk algebras characterizes the modal logic S4/GRZ.

The condition (Grz) is rather difficult for ``intuitive understanding''. Nevertheless, some semantical characterizations were found. It is known that modal frame validates $grz$ iff it is Noetherian partially ordered set \cite[Theorem 3.38]{CZ97} (however see  \cite{Jer04} for set theoretical subtleties). Zakharyaschev showed that a general transitive frame validates $grz$ iff it is not subreducible to a modal frame which consists of one irreflexive  point, and to a modal frame which consists of two points with the total relation (two-element cluster) \cite[Proposition 9.3]{CZ97}, see also \cite{Str17} for an algebraic proof. Moreover, the class of Grzegorczyk algebras is an intersection of two splitting subvarieties of interior algebras. More precisely, Blok proved in \cite{Blo76} that an interior algebra is a Grzegorczyk algebra iff it does not have a subalgebra admitting a homomorphisms onto $\S_2$ or $\S_{1,2}$ (Corollary \ref{prop::characterization of Grzegorczyk algebras}). These algebras are depicted in Figure \ref{fig:: S_2 and S_1,2} (open elements are depicted by $\mBox$). 

\begin{figure}[h]
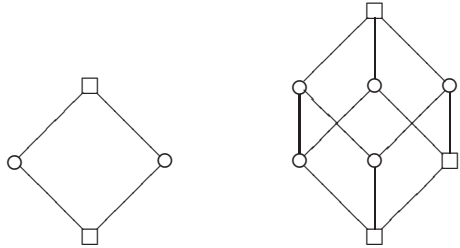

\[
\xy
(10,20)*{\text{\raisebox{0.9pt}{$\Box$}}}="AC";
(0,9.8)*{\piccirc}="A";  (20,10)*{\piccirc}="C";
(10,0)*{\text{\raisebox{0.9pt}{$\Box$}}}="ZERO";
(9.07,1.07); (.7,9.44) **\dir{-}; 
(10.93,1.07);  (19.35,9.49) **\dir{-}; 
(9.07,19.22); (.55,10.7) **\dir{-}; 
(19.32,10.82); (10.9,19.24) **\dir{-};  
\endxy \quad\quad\quad\quad
\xy
(10,30)*{\text{\raisebox{0.9pt}{$\Box$}}}="ONE";
(0.01,19.81)*{\piccirc}="AB"; (10,20)*{\piccirc}="AC"; (20,20)*{\piccirc}="BC";
(0,10)*{\piccirc}="A"; (10,10)*{\piccirc}="B"; (20,10)*{\text{\raisebox{0.9pt}{$\Box$}}}="C";
(10,0)*{\text{\raisebox{0.9pt}{$\Box$}}}="ZERO";
(9.07,1.07); (.7,9.44) **\dir{-}; 
(10,1.1); (10,9.2) **\dir{-}; 
(10.93,1.07);  (19.1,9.24) **\dir{-}; 
"A"; (0,19.0) **\dir{-}; 
(9.37,19.52); (.60,10.75) **\dir{-}; 
(9.37,10.78); (0.60,19.55) **\dir{-}; 
(19.37,19.52); (10.60,10.75) **\dir{-}; 
(19.07,11.07); (10.7,19.44) **\dir{-};  
(20,11.1); (20,19.2) **\dir{-}; 
(9.07,29.22); (.60,20.75) **\dir{-}; 
"AC"; (10,29.2) **\dir{-}; 
(19.37,20.78); (10.95,29.20) **\dir{-}; 
\endxy
\]
\caption{Interior algebras $\S_2$ and $\S_{1,2}$.}
\label{fig:: S_2 and S_1,2}
\end{figure}

It appears that the variety of Grzegorczyk algebras is generated, even as a universal class, by the class of all interior algebras which are generated by open elements. Such algebras are  isomorphic to free Boolean extensions of Heyting algebras, see Section \ref{sec:: the Blok-Esakia theorem}. In this section we provide a new proof of the fact that every interior algebra which is generated by open elements is a Grzegorczyk algebra (Corollary \ref{cor:: openly generated => Grzegorczyk}). The fact that the universal class generated by the free Boolean extensions of Heyting algebras is not a proper subclass of Grzegorczyk algebras follows from the Blok lemma. (Alternatively, one may use filtration and the fact that all finite Grzegorczyk algebras are generated by open elements. But we do not present this approach here). As a byproduct, we obtain new algebraic proofs of the above mentioned  Blok's characterization of Grzegorczyk algebras.

Let us start with recalling a crucial notion of stable homomorphisms. Let $\M$ and $\N$ be modal algebras. We say that a mapping $f\colon N\to M$ is a \emph{stable homomorphism} from $\N$ into $\M$ if it is a Boolean homomorphism and
\[
f(\mBox a)\leq \mBox f(a)
\]
holds for every $a\in M$. The reader may consult e.g. \cite{BBI16,Ghi10} for the importance of stable homomorphisms in modal logic (in \cite{Ghi10} they are called \emph{continuous morphism}). Note that a stable homomorphism does not need to be a homomorphism. For instance $\S_2$ does not admit a homomorphism onto a two-element interior algebra, but it admits such stable homomorphisms.


\begin{lemma}\label{fact::stable homo}
Let $\M$  be an interior algebra and $\S$ be a simple interior algebra. Then the mapping $f\colon M\to S$ is a stable homomorphism if and only if it is a Boolean homomorphism and for every $a\in M$ we have $f(\mBox a)\in\{0,1\}$.
\end{lemma}

\begin{proof}
Assume that $f$ is a stable homomorphism. By the definition, $f$ is a Boolean homomorphism. Moreover, if $f(\mBox a)<1 $ then $f(\mBox a)=f(\mBox\mBox a)\leq\mBox f(\mBox a)=0$.
For the opposite implication assume that $f$ is a Boolean homomorphism and $f(\mBox a)=\{0,1\}$. If $f(\mBox a)=0$ then, clearly, $f(\mBox a)\leq \mBox f(a)$. If $f(\mBox a)=1$ then, since $\mBox a\leq a$, $f(a)=1$, and so $f(\mBox a)=1=\mBox f(a)$.
\end{proof}

\begin{proposition}\label{thm:: main characterization Grzegorczyk}
Let $\M$ be an interior algebra and $a\in M$. If {\rm (Grz)} for $a$ in $\M$, then there exists a stable homomorphism $h$ from $\M$ onto $\S_2$ such that $h(a)$ is a (co)atom in $\S_2$.
\end{proposition}

\begin{proof}
Let us consider the term
\[
t(x)=\mBox(x\to\mBox x)\to x.
\]
Let $G$ be a maximal open filter of $\M$ with respect to the following conditions
\[
a\not\in G \quad \text { and }\quad t(a)\in G.
\]
Since the inequality (Grz) does not hold, the Zorn lemma guarantees the existence of $G$. Let $\M'=\M/G$, $g\colon \M\to\M';\; d\mapsto d/G$, and  $a'=g(a)=a/G$. Then (Grz) does not hold in $\M'$ for $a'$ either. Furthermore,
$\M'$ is subdirectly irreducible and a least nontrivial open filter of $\M$ is the open filter generated by $a'$. It follows that all open non-top elements in $\M'$ are in the interval $[0,\mBox a']$.

Let $\S$ be the interior algebra constructed in the following way. Its Boolean reduct is a quotient of Boolean reduct of  $\M'$ divided by the Boolean filter $F$ generated by $\neg \mBox a'$. We equip $\S$ with only two open elements: the top and the bottom ones. This means that $\mBox b=0$ iff $b<1$ and $\mBox 1=1$ in $\S$. Let $f\colon M'\to S;\; b\mapsto b/F$. Clearly, $f$ is a Boolean surjective homomorphism. Moreover, $f^{-1}(0)=[0,\mBox a']$. Hence Lemma \ref{fact::stable homo} yields that $f$ is a stable homomorphism from $\M'$ onto $\S$.

Let $\mBox a'\leq b,c$ and $b\neq c$. Then $\mBox a'\leq b\leftrightarrow c<1$ and hence $\neg \mBox a'\not\leq b\leftrightarrow c$. Thus the pair $(b,c)$ is not in the Boolean congruence associated with $F$. This shows that restriction of $f$ to the interval $[\mBox a',1]$ is injective. We prove that this interval has more than two elements.

For every $b\in M'$ we have
\[
 t(\mBox b)=\mBox b.
\]
But in $\M'$ we have $t(a')=1>a'$. Thus $a'$ is not open and hence
\[
\mBox a' < a' <1.
\]
This and the fact that the restriction of $f$ onto $[\mBox a',1]$ is injective yields that $\S$ has at least four elements. In particular, there exists a Boolean homomorphism $k$ from the Boolean reduct of $\S$ onto the Boolean reduct of $\S_2$. Clearly, $k$ is also a stable homomorphism. Also, we may choose $k$ such that $kfg(a)$ is a (co)atom in $\S_2$.

Finally, a composition of stable surjective homomorphisms is a stable surjective homomorphism. Thus the composite mapping
\[
kfg
\]
is a stable homomorphism from $\M$ onto $\S_2$.
\end{proof}

The converse of Proposition \ref{thm:: main characterization Grzegorczyk} 
does not hold.  Indeed, let $\M$ be an interior algebra where the carrier $M$ is the power set of the set $\Nat$ of natural numbers, the Boolean operations are the set theoretic operations, and
\[
\mBox a=
\begin{cases}
1 & \text{ if } a=1 \\
\{0,\ldots,k-1\} & \text{ if } a\neq 1\text{ and } k=\min\neg a
\end{cases}.
\]
Note that $\M$ is the dual algebra for the modal frame $(\Nat,\geq)$ in the J\'onsson-Tarski duality. Then $\M$ is a Grzegorczyk algebra admitting a stable homomorphism onto $\S_2$, see \cite[Example 3.5]{Str17} for details. However, the converse of Proposition \ref{thm:: main characterization Grzegorczyk} holds for finite interior algebras. Indeed, it follows from Zakharyaschev's characterization of Grzegorczyk frames, see also \cite[Theorem 2.4]{Str17}.

\begin{corollary}\label{cor:: openly generated => Grzegorczyk}
Let $\M$ be an interior algebra generated by its open element. Then $\M$ is a Grzegorczyk algebra.
\end{corollary}

\begin{proof}
Let $f$ be a stable homomorphism from $\M$ into $\S_2$. By Lemma \ref{fact::stable homo}, for every $a\in M$ we have $f(\mBox a)\in \{0,1\}$. Since $f$ preserves Boolean operations, $\{0,1\}$ is closed under Boolean operations and in fact the algebra $\M$ is also Boolean generated by its open elements, $f$ maps $M$ onto $\{0,1\}$. Thus $f$ is not surjective.
\end{proof}

Let us also show that
the proof of Proposition \ref{thm:: main characterization Grzegorczyk} may be slightly modified in order to obtain a new proof of Blok's characterization for Grzegorczyk algebras.

\begin{corollary}[\protect{\cite[Example III.3.9]{Blo76}}]
\label{prop::characterization of Grzegorczyk algebras}
An interior algebra $\M$ is a Grzegorczyk algebra if and only if it does not have a subalgebra with a homomorphic image isomorphic to  $\S_2$ or to $\S_{1,2}$.
\end{corollary}

\begin{proof}
Assume that $\M$ is an interior algebra and $a\in M$ are such that (Grz) does not hold for $a$. Let $\M'$ and $a'$ be as in the proof of Proposition \ref{thm:: main characterization Grzegorczyk}. Let $\N$ be a subalgebra of $\M'$ generated by $a'$. We claim that $\N$ is isomorphic to $\S_2$ or to $\S_{1,2}$.

If $\mBox a'=0$, then $\N$ is simple and hence, since it is generated by $a'$, is isomorphic to $\S_2$. So let us assume that
\[
0<\mBox a'<a'<1.
\]
Let $P$ be the carrier of the eight-element Boolean algebra generated by this chain.
This means that
\[
P=\{
a',\mBox a', \neg a', \neg\mBox a', \neg a\join \mBox a', a\meet \neg\mBox a',0,1\}.
\]
In order to show that $\N$ is isomorphic to $\S_{1,2}$ it is enough to verify that $\mBox \neg \mBox a' =0$ and $\mBox(\neg a'\join \mBox a')=\mBox a'$. This, in particular, would show that $P$ is closed under the $\mBox$ operation and, since $P$ is closed under the Boolean operations, $P=N$.

Here we will use the fact that $\mBox a'$ is the largest not equal to 1 open element in $\N$. So
\[
\mBox \neg \mBox a' = \mBox \neg \mBox a' \meet \mBox a'=\mBox \neg \mBox a'\meet \mBox\mBox a'=\mBox (\neg \mBox a' \meet \mBox a')=\mBox 0=0
\]
and
\[
\mBox(\neg a'\join \mBox a')=\mBox(\neg a'\join \mBox a')\meet\mBox a'=
\mBox((\neg a'\join \mBox a')\meet a')=\mBox \mBox a'=\mBox a'.
\]

For the converse, note that (Grz) fails for any coatom in $\S_2$ and in $\S_{1,2}$. Thus it fails for some element in every algebra having subalgebra with a homomorphic image isomorphic to $\S_2$ or to $\S_{1,2}$.

\end{proof}

\section{The Blok-Esakia theorem}\label{sec:: the Blok-Esakia theorem}

The connection of Heyting algebra with interior algebra is given by the following McKinsey-Tarski theorem \cite[Section 1]{MT46} (see also \cite[Chapter 1]{Blo76}, \cite[Theorem 2.2]{BD75} and \cite[Section 3]{MR74}).  Recall that open elements of an interior algebra $\M$ form the Heyting algebra $\sf O(\M)$ with the order structure inherited from $\M$.

\begin{theorem}\label{prop:: McKinsey-Tarski thm}
For every Heyting algebra $\H$ there is an interior algebra ${\sf B}(\H)$ such that
\begin{enumerate}
\item ${\sf OB}(\H)=\H$;
\item for every interior algebra $\M$, if
$\H\leq{\sf O}(\M)$, then ${\sf B}(\H)$ is isomorphic to the subalgebra of $\M$ generated by $H$;
\item for every interior algebra $\M$ and every
    homomorphism $f\colon\H\to{\sf O}(\M)$ there is a unique homomorphism $\bar f\colon{\sf B}(\H)\to\M$ extending $f$.
\end{enumerate}
\end{theorem}
The algebra ${\sf B}(\H)$ is called the \emph{free Boolean extension of} $\H$. We will treat $\sf O$ and $\sf B$ as a class operators. This is, for a class $\K$ of Heyting algebras and a class $\mathcal M$ of interior algebras we put ${\sf B}(\K)=\{{\sf B}(\H)\mid \H\in\K\}$ and  ${\sf O}(\mathcal M)=\{{\sf O}(\M)\mid \M\in\mathcal M\}$. From Theorem \ref{prop:: McKinsey-Tarski thm}, we immediately obtain
\[
{\sf OB}(\K)=\K\quad\quad\text{ and }\quad\quad{\sf BO}(\mathcal M)\subseteq {\sf S}(\mathcal M).
\]
Let us list their basic properties with respect to other class operators: $\sf H$ homomorphic image, $\sf S$ subalgebra, $\sf P$ product, and $\sf P_U$ ultraproduct class operators. (We tacitly
assume that all class operators are composed with the isomorphic image class operator, i.e, $\sf S=IS$ and so on.)

\begin{lemma}\label{lem:: basic properties of B and O}
Let $\K$ be a class of Heyting algebras and $\mathcal M$ be a class of interior algebras. Then
\begin{align*}
{\sf OH}(\mathcal M)&={\sf HO}(\mathcal M),&{\sf BH}(\K)&={\sf HB}(\K),\\
{\sf OS}(\mathcal M)&={\sf SO}(\mathcal M),&{\sf BS}(\K)&={\sf SB}(\K),\\
{\sf OP}(\mathcal M)&={\sf PO}(\mathcal M),&{\sf BP}(\K)&\subseteq{\sf SPB}(\K),\\
{\sf OP_U}(\mathcal M)&={\sf P_UO}(\mathcal M),&{\sf BP_U}(\K)&\subseteq{\sf SP_UB}(\K).
\end{align*}
\end{lemma}

\begin{proof}
The equalities from the first column follows directly from the definitions. Let us illustrate this by verifying the fourth one. Let $\M=\prod_i\M_i/U$ and $\H=\prod_i \H_i/U$ be ultraproducts of interior algebras and Heyting algebras respectively, where  $\H_i={\sf O}(\M_i)$. Let $\alpha$ and $\beta$ be congruences of $\prod_i\M_i$ and $\prod_i\H_i$ respectively induced by the ultrafilter $U$. Then $\beta=\alpha\cap (\prod_i H_i)^2$. It follows that the mapping $f\colon c/\beta\mapsto c/\alpha$ is an embedding of $\H$ into ${\sf O}(\M)$. In order to see that $f$ is surjective, note that for $c\in \prod_i M_i$ we have the equivalences: $c/\alpha$ is open iff $\{i\mid c_i=\mBox c_i\}\in U$ iff $c\,\alpha\,c'$, where $c_i'=c_i$ if $\mBox c_i=c_i$ and $c'_i=1$ otherwise. Hence,  if $c/\alpha$ is open, then $c/\alpha=c'/\alpha=f(c'/\beta)$.

The equalities from the second column follows from Theorem \ref{prop:: McKinsey-Tarski thm}. For the first containment, note that
\[
{\sf BP }(\K)={\sf BPOB}(\K)={\sf BOPB}(\K)\subseteq{\sf SPB}(\K).
\]
Indeed, here the first equality and the containment follows from the properties of $\sf B$ and $\sf O$ recalled before the lemma, and the second equality follows from the third equality in the first column.
The proof of the second inclusion is analogical.
\end{proof}

\vspace{.5em}

Let $\Hey$ be the variety of all Heyting algebras, and $\Grz$ be the variety of all Grzegorczyk algebras. Let $\mathcal L_V(\Hey)$ and $\mathcal L_V(\Grz)$ denote the lattices of varieties of Heyting algebras and Grzegorczyk algebras respectively. Recall two operators on these lattices:
\begin{align*}
\rho&\colon\mathcal L_V(\Grz)\to\mathcal L_V(\Hey);\; \W\mapsto{\sf O}(\W)\\
\sigma&\colon\mathcal L_V(\Hey)\to\mathcal L_V(\Grz);\;\V\mapsto{\sf VB}(\V),
\end{align*}
where ${\sf V(\K)}$ denotes a least variety containing $\K$.
(The analog of this operator for logics was introduced in \cite{MR74}.)
Notice that $\rho$ is well defined, i.e., when $\W$ is a variety of interior algebras, then $\rho(\W)$ is indeed the variety of Heyting algebras. Indeed, this follows from Lemma \ref{lem:: basic properties of B and O}. A much less obvious fact is that if $\V$ is a variety of Heyting algebras, then $\sigma(\V)$ is a variety of Grzegorczyk algebras. This fact follows from the following Proposition (based on the previous section).

\begin{proposition}[\protect{\cite[Corollary III.7.9]{Blo76}}]\label{prop:: B(H) is Grzegorczyk}
Let $\H$ be a Heyting algebra. Then ${\sf B}(\H)$ is a Grzegorczyk algebra.
\end{proposition}

\begin{proof}
By Theorem \ref{prop:: McKinsey-Tarski thm} the class of interior algebras isomorphic to algebras of the form ${\sf B}(\H)$ coincides with the class of interior algebras generated by its open elements. Thus the assertion  follows from Corollary \ref{cor:: openly generated => Grzegorczyk}.
\end{proof}

Let us now formulate the Blok-Esakia theorem \cite[Theorem 7.10]{Blo76}, \cite[Theorem 7.11]{Esa79a} (see aslo \cite[Section 2]{Bez09}, \cite[Theorem 9.66]{CZ97}, \cite[Section 3]{WZ14}) in algebraic terms.

\begin{B-E-theorem}
The mappings $\sigma$ and $\rho$ are mutually inverse isomorphisms between the lattices $\mathcal L_V(\Hey)$ and $\mathcal L_V(\Grz)$.
\end{B-E-theorem}

The hardness of the algebraic proof of the Blok-Esakia theorem lies in in the fact that not all Grzegorczyk algebras are representable as ${\sf B}(\H)$. The following lemma, doe to Blok, shows how to overcome this difficulty. (The original, rather technical formulation, may be found in \cite[Lemma III.7.6]{Blo76}, \cite[Lemma 2]{WZ14}. Our presentation is simpler in use.)

\begin{B-lemma}\label{lem:: Blok lemma}
Let $\M$ be a Grzegorczyk algebra. Then $\M$ embeds into some elementary extension of ${\sf BO}(\M)$.
\end{B-lemma}

Our formulation follows from the original one. However, for the seek of completeness, we provide its proof in Appendix. Note that it is essentially the same proof. It is just written in a bit different fashion. Here we provide a simple proof for the finite case \cite[Theorem II.2.11]{Blo76}. We failed to find a semantical transparent proof for the general case.

\begin{B-lemma-fc}
Let $\M$ be a finite Grzegorczyk algebra. Then $\M$ is isomorphic to ${\sf BO}(\M)$.
\end{B-lemma-fc}

\begin{proof}
Let us first prove that if $a$ is an element in $M$ and $\mBox a<1$, then either $1$ covers $\mBox a$ or there exists an element $b$ in $M$ such that $\mBox a<\mBox b<1$. (This holds for an arbitrary Grzegorczyk algebras.)  So assume that it is not the case, i.e., that there exists an interval $[\mBox a, 1]$ in $\M$ with more than $2$ elements and with exactly two open elements $\mBox a$ and $1$. We may assume that $\mBox a<a$. Then (Grz) fails in $\M$ for $a$. Indeed, since $\mBox a<a<1$,
\[
\mBox a<\neg a\join \mBox a <1.
\]
Hence $\mBox (a\to\mBox a)=\mBox a$ and
\[
\mBox (\mBox (a\to\mBox a)\to a)=\mBox 1=1\not\leq a.
\]
This and the finiteness of $|M|$ yields that either $\M$ is trivial or has an open coatom. We use this fact to show that there is a maximal chain in $\M$ whose all elements are open. This and the finiteness then imply that $\M$ is generated by open elements.

We proceed by induction on the cardinality of $M$.
If $|M|=1$ then $\M$ has only one element which is open. Thus the thesis holds. So assume that $\M$ is a finite nontrivial Grzegorczyk algebra and the thesis holds for all smaller Grzegorczyk algebras. Let $c$ be an open coatom in $\M$ and let $F=\{1,c\}$. Then $F$ is an open filter. By the induction assumption, there is a maximal chain  $\{a_1/F,\ldots,a_k/F\}$ in $\M/F$ consisting of open elements in $\M/F$. Assume that each $a_i$ is chosen as a minimal element in $a_i/F$, i.e., $a_i\leq c$. Then $a_i/F=\{a_i,a_i\join\neg c\}$ and at least one of this two elements is open. Actually, every $a_i$ is open, since if $a_i\join \neg c$ is open then $\mBox a_i=\mBox((a_i\join \neg c)\meet c)=\mBox(a_i\join \neg c)\meet \mBox c=(a_i\join \neg c)\meet  c=a_i$. Thus $\{1,a_1,\ldots,a_k\}$ is a maximal chain in $\M$ consisting of open elements.
\end{proof}

Let $\mathcal L_U(\Hey)$ and $\mathcal L_U(\Grz)$ be the lattices of all universal classes of Heyting algebras and Grzegorczyk algebras respectively. Let us define two operators on these lattices. For $\U\in \mathcal L_U(\Hey)$ and $\mathcal Y\in \mathcal L_U(\Grz)$ put
\begin{align*}
\rho'&\colon\mathcal L_U(\Grz)\to\mathcal L_U(\Hey);\; \mathcal Y\mapsto\{\sf O(\M)\mid \M\in\mathcal Y\}\\
\sigma'&\colon\mathcal L_U(\Hey)\to\mathcal L_U(\Grz);\;\U\mapsto\sf U(\{B(\H)\mid \H\in \U\}),
\end{align*}
where ${\sf U}=\sf SP_U$. Recall that if $\K$ is a class of algebras in the same signature, then ${\sf SP_U}(\K)$ is a least universal class containing $\K$ \cite[Theorem V.2.20]{BS81}.
Since $\sf O$ commutes with $\sf P_U$ and $\sf S$ (Lemma \ref{lem:: basic properties of B and O}), $\rho(\mathcal Y)$ is an universal class for $\mathcal Y\in \mathcal L_U(\Grz)$. Clearly, thus defined $\rho'$ is an extension of previously defined operator $\rho$ for varieties. Note also that, by Proposition \ref{prop:: B(H) is Grzegorczyk}, $\sigma'(\mathcal Y)\in \mathcal L_U(\Grz)$ for $\mathcal Y\in\mathcal L_U(\Hey)$.
It follows from Lemma \ref{lem:: preservation of quasivarieties} that $\sigma'$ is also an extension of $\sigma$. 

\begin{lemma}\label{lem::easy part of B-E thm}
Let $\U\in\mathcal L_U(\Hey)$. Then $\rho'\sigma'(\U)=\U$. In particular,  $\M\in\sigma'(\U)$ implies ${\sf O}(\M)\in \U$.
\end{lemma}

\begin{proof}
By the equalities from Lemma \ref{lem:: basic properties of B and O} and Theorem \ref{prop:: McKinsey-Tarski thm} point (1), we have
\[
\rho'\sigma'(\U)={\sf OSP_UB}(\U)={\sf SP_UOB}(\U)={\sf SP_U}(\U)=\U.
\]
\end{proof}

\begin{lemma}\label{lem:: preservation of quasivarieties}
Let $\K$ be a universal class of Heyting algebras. Then $\K$ is a quasivariety/universal positive class/variety if and only if  $\sigma'(\K)$ is quasivariety/universal positive class/variety respectively. In particular, $\sigma'(\V)=\sigma(\V)$ for every variety $\V$ of Heyting algebras. 
\end{lemma}

\begin{proof}
Recall first that a class of algebras  closed under isomorphic images is a quasivariety iff it is closed under $\sf S$, $\sf P$ and $\sf P_U$  \cite[Theorem V.2.25]{BS81}, is a universal positive class iff it is closed under $\sf S$, $\sf H$ and $\sf P_U$ \cite[Corollary 2 of Theorem 4 in Section 43]{Gra08},  a variety if  is closed under $\sf S$, $\sf H$ and $\sf P$ \cite[Theorem II.11.9]{BS81}.

Thus the backward implication follows from Lemmas \ref{lem:: basic properties of B and O} and \ref{lem::easy part of B-E thm}.

For the forward implication it is enough to show that if $\K$ is a universal class closed under $\sf P$ or $\sf H$, then $\sigma'(\K)$ is also closed under $\sf P$ or $\sf H$ respectively.

Assume first that $\K$ is closed under the $\sf P$ operator.
Let $\M_i\in \sigma'(\K)$ and $\M=\prod \M_i$. By Lemma \ref{lem::easy part of B-E thm} and the assumption,  ${\sf O}(\M)=\prod {\sf O}(\M_i)\in \K$. Thus by the Blok lemma, $\M\in\sigma'(\K)$.

Assume now that $\K$ is closed under the $\sf H$ operator.
Let $h\colon\N\to\M$, were $\N\in \sigma'(\K)$, be a surjective homomorphism. Let ${\sf O}(h)$ be the restriction of $h$ to the carrier of ${\sf O}(\M)$. Then ${\sf O}(h)\colon{\sf O}(\N)\to{\sf O}(\M)$ is also a surjective homomorphism. By Lemma \ref{lem::easy part of B-E thm}, ${\sf O}(\N)\in\K$, and by the assumption, ${\sf O}(\M)\in\K$. Hence the Blok lemma yields that $\M\in \sigma'(\K)$.
\end{proof}

From now on, we notationally identify $\sigma'$ with $\sigma$ and $\rho'$ with $\rho$.

\begin{B-E-theorem-ex}[\protect{\cite[Theorem 5.5]{Jer09}}]\label{thm:: B-E theorem extended}
The mappings $\sigma$ and $\rho$ are mutually inverse isomorphisms between the lattices $\mathcal L_U(\Hey)$ and $\mathcal L_U(\Grz)$.
\end{B-E-theorem-ex}

\begin{proof}
For $\mathcal Y\in\mathcal L_U(\Grz)$ the inclusion $\mathcal Y\subseteq\sigma\rho\,(\mathcal Y)$ follows from the Blok lemma. The inclusion $\sigma\rho\,(\mathcal Y)\subseteq\mathcal Y$ follows from Theorem \ref{prop:: McKinsey-Tarski thm} point (2) and the fact that $\mathcal Y$ is closed under taking subalgebras.
For $\U\in\mathcal L_U(\Hey)$ the equality $\U=\rho\sigma\,(\mathcal U)$ follows from Lemma \ref{lem::easy part of B-E thm}.
\end{proof}

Let $\mathcal L_Q(\Hey)$/$\mathcal L_{U^+}(\Hey)$ and $\mathcal L_Q(\Grz)$/$\mathcal L_{U^+}(\Grz)$ be lattices of quasivarieties/universal positive classes of Heyting algebras and Grzegorczyk algebras respectively.

\begin{B-E-theorem-res}[\protect{\cite[Theorem 5.5]{Jer09}}]\label{thm:: B-E theorem restricted}
The appropriate restrictions of $\sigma$ and $\rho$ are mutually inverse isomorphisms between the lattices $\mathcal L_U(\Hey)$/$\mathcal L_{U^+}(\Hey)$/$\mathcal L_Q(\Hey)$/$\mathcal L_V(\Hey)$ and the lattices $\mathcal L_U(\Grz)$/$\mathcal L_{U^+}(\Grz)$/$\mathcal L_Q(\Grz)$/$\mathcal L_V(\Grz)$ respectively.
\end{B-E-theorem-res}

\begin{proof}
It follows from the Blok-Esakia theorem extended version and Lemma \ref{lem:: preservation of quasivarieties}.
\end{proof}

\section{Strong structural completeness}

Structural completeness and universal completeness are  well established properties for single-conclusion consequence relations and quasivarieties. However its extension to multi-conclusion consequence relations and to universal classes is recent \cite{Iem15}. Here we focus on a connected  new properties of strong structural completeness and strong universal completeness. For dealing algebraically with the structural and universal completenesses for mcrs we need to introduce a new notion of free families. It is done in a separate paper \cite{Str16}.    

We prove that the Blok-Esakia isomorphism for universal preserves  and reflects strong structural and universal completenesses. For (quasi)-varieties strong versions are equivalent to standard ones. For varieties  preservation and reflection of structural completeness was proved in \cite[Theorem 5.4.7]{Ryb97}. In our opinion, the proof presented here is simpler and more general.

For an mcr $\vdash$ and an inference rule $r$ let $\vdash_r$ be a least mcr extending $\vdash$ and containing $r$.
An inference rule $r$ is \emph{weakly admissible} for $\vdash$ if 
\[
\{\psi\in Form\mid \0\vdash\{\psi\}\}=\{\psi\mid \0\vdash_r\{\psi\}\},
\]
and \emph{admissible} for $\vdash$ if 
\[
\{\Delta \text{ a finite subset of }Form  \mid \0\vdash\Delta\}=\{\Delta\mid \0\vdash_r\Delta\}.
\]
In other words, $r$ is (weakly) admissible if $\vdash$ and $\vdash_r$ share the same (multi-)theorems.

A mcr is \emph{structurally/universally complete} if the sets of its admissible and derivable single-conclusion/multi-conclusion inference rules coincide. And a mcr is \emph{strongly structurally/universally complete} if the sets of its weakly admissible and derivable single/multi-conclusion inference rules coincide.

Let us summarize these properties in the following table. 

\begin{center}
\begin{tabular}{ |c| c| c| }
\hline
  & weakly admissible & admissible \\ 
 \hline
 single-conclusion & strong structural completeness & structural completeness \\  
 \hline
 multi-conclusion & strong universal completeness & universal completeness \\  
\hline
\end{tabular}
\end{center}

\mbox{}

All these notions have algebraic counterparts. 
A universal disjunctive sentence $v$ is \emph{weakly admissible} for a universal class $\U$ if the sets of identities satisfied in $\U$ and in the class of all algebras from $\U$ satisfying $v$ coincide. And $v$ is \emph{admissible} for $\U$ if the sets of positive universal disjunctive sentences satisfied in $\U$ and in the class of all algebras from $\U$ satisfying $v$ coincide.

 Analogically, a universal class $\U$ is \emph{{\rm (}strongly{\rm )}  structurally complete} if the set of (weakly) admissible for $\U$ quasi-identities and the set of quasi-identities which hold in $\U$  coincide.
And $\U$ is \emph{{\rm (}strongly{\rm )} universally complete} if the set of (weakly) admissible for $\U$ disjunctive universal sentences and the set of disjunctive universal sentences which hold in $\U$  coincide. The following fact follows from Theorem \ref{thm:: completeness}.

\begin{fact}
Let $\vdash$ be a modal or intermediate mcr, $\U$ be a universal class of modal or Heyting algebras respectively, and let $r$ be an inference rule in the language of $\vdash$. Assume that $\vdash$ and $\U$ correspond to each other (in the sense of Theorem \ref{thm:: completeness}). Then $r$ is (weakly) admissible for $\vdash$ if and only if its translation ${\sf T}(r)$ is (weakly) admissible for $\U$. Consequently, $\vdash$ is (strongly) structurally/universally complete if and only if $\U$ is (strongly) structurally/universally complete. 
\end{fact}


We recall the notion of free algebras. Let $V$ be a set. Let $\Fr$ be an algebra such that $V\subseteq F$ and $\K$ be an arbitrary class of algebra in the same signature as $\Fr$. We say that $\Fr$ has \emph{the universal mapping property for $\K$ over $V$} if for every algebra $\A\in \K$ and every mapping $f\colon V\to A$ there exists a unique homomorphism $f'\colon \Fr\to \A$ such that $f'|_V=f$. Clearly, such algebra $\Fr$ is not uniquely determined. Moreover, it does not need to be the case that $\Fr\in\K$. Therefore additional condition is considered. An algebra $\Fr$ is \emph{free for $\K$ over $V$} if it has the universal mapping property for $\K$ over $V$ and $\Fr$ belongs to the variety generated by $\K$. Then for every nonempty set $V$ and every class $\K$ of algebras in the same language containing a nontrivial algebra there exists a free algebra $\Fr$ for $\K$ over $V$. Moreover, an algebra $\Fr$ is uniquely determined up to isomorphism whose restriction to $V$ is an identity.


For a class $\K$ of algebras in a fixed language, we write ${\sf Q}(\K)$ and ${\sf V}(\K)$ to denote a least quasivariety and a least variety respectively containing $\K$. In case when $\K$ consist of one algebra $\bf A$ we also write ${\sf Q}({\bf A})$ and ${\sf V}({\bf A})$.
Let us summarize the needed properties of free algebras in the following fact.

\begin{fact}\label{fact:: free algebras}
Let $\Fr$ be free for $\K$ over a nonempty set $V$. Then
\begin{enumerate}
\item $\Fr\in{\sf Q}(\K)$ {\rm (}\cite[Theorem II.10.12]{BS81}{\rm )};
\item $\Fr$ is also free over $V$ for ${\sf V}(\K)$ {\rm (}\cite[Corollary II.11.10]{BS81}, see also \cite[Proposition 4.8.9]{Kra99} for a more direct argument{\rm )};
\item if $V$ is infinite and $(\forall\bar x)\,e(\bar x)$ is an identity, then $\K\models (\forall\bar x)e(\bar x)$ if and only if $\Fr\models (\forall\bar x)\,e(\bar x)$ if and only if $\Fr\models e(\bar v)$, where $\bar v$ is a tuple of mutually distinct elements from $V$ of the same length as $\bar x$.  
  {\rm (}\cite[Theorem II.11.4]{BS81}{\rm )}.
\end{enumerate}
\end{fact}

For a sentence $\varphi$ let us denote by $\Mod(\varphi)$ the class of all algebras, in the same fixed language as $\varphi$, satisfying $\varphi$. 

\begin{proposition}\label{prop:: w-admissibility}
Let $\U$ be a universal class, $\Fr$ be a free algebra for $\U$ over an infinite set, and $v$ be a universal sentence. Then the following conditions are equivalent
\begin{enumerate}
\item  $v$ is weakly admissible for $\U$;
\item ${\sf V}(\U\cap\Mod(v))={\sf V}(\U)$;
\item $\Fr\in {\sf Q}(\U\cap\Mod(v))$. 
\end{enumerate}
\end{proposition}

\begin{proof}
The definition of weak admissibility may be stated as
\[
\{e\text{ is an identity}\mid \U\cap \Mod(v)\models e\}\subseteq \{e\text{ is an identity}\mid \U\models e\}.
\]
This shows that (1) is equivalent to (2). By Fakt \ref{fact:: free algebras} point (3), the condition (2) is equivalent to
\[
\Fr\in{\sf V}(\U\cap \Mod(v)).
\]
Since $\Fr$ is free for $\U$, it has the universal mapping property for $\U\cap\Mod(v)$. Thus the last condition is equivalent to $\Fr$ being free for $\U\cap\Mod(v)$. Thus Fact \ref{fact:: free algebras} point (1) yields that it is equivalent to (3).
\end{proof}


\begin{proposition}\label{prop:: SSC}
Let $\U$ be a universal class and $\Fr$ be a free algebra for $\U$ over an infinite set. Then the following conditions are equivalent
\begin{enumerate}
\item $\U$ is strongly structurally complete;
\item for every quasi-identity $q$
\[
\Fr\in {\sf Q}(\U\cap\Mod(q)) \text{ yields } \U\models q;
\]
\item for every subquasivariety $\Q$ of ${\sf Q}(\U)$
\[
\Fr\in {\sf Q}(\U\cap\Q) \text{ yields } \U\subseteq \Q.
\]
\end{enumerate}
\end{proposition}

\begin{proof}
The equivalence (1)$\Leftrightarrow$(2) follows directly from Proposition \ref{prop:: w-admissibility} and the definition of strong structural completeness.

The implication (3)$\Rightarrow$(2) may be obtained by considering the quasivariety $\Q={\sf Q}(\U)\cap \Mod(q)$. 

For the implication (2)$\Rightarrow$(3), let us assume that $\Fr\in {\sf Q}(\U\cap\Q)$. Let $\Sigma$ be the set of all quasi-identities satisfied in $\Q$. Then for every $q\in \Sigma$, we have $\Fr\in {\sf Q}(\U\cap\Mod(q))$. Thus (2) yields 
$
\U\subseteq \bigcap_{q\in\Sigma}\Mod(q)=\Q
$.  
\end{proof}

Similarly, one may prove the following fact.

\begin{proposition}\label{prop:: SUC}
Let $\U$ be a universal class and $\Fr$ be a free algebra for $\U$ over an infinite set. Then the following conditions are equivalent
\begin{enumerate}
\item $\U$ is strongly universally complete;
\item for every disjunctive universal sentence $v$
\[
\Fr\in {\sf Q}(\U\cap\Mod(v)) \text{ yields } \U\models v;
\]
\item for every  universal  subclass $\mathcal W$ of $\U$
\[
\Fr\in {\sf Q}(\mathcal W) \text{ yields } \U=\mathcal W.
\]
\end{enumerate}
\end{proposition}

\begin{lemma}\label{lem:: sigmaV=VB sigmaQ=QB}
Let $\K$ be a class of Heyting algebras. Then
\[
\sigma{\sf V}(\K)={\sf VB}(\K)\quad\text{and}\quad\sigma{\sf Q}(\K)={\sf QB}(\K).
\]
Hence, if $\U$ is a universal class of Heyting algebras, then
\[
\sigma{\sf V}(\U)={\sf V}\sigma(\U)\quad\text{and}\quad\sigma{\sf Q}(\U)={\sf Q}\sigma(\U).
\]

\end{lemma}

\begin{proof}
Recall that ${\sf V}(\K)={\sf HSP}(\K)$ and ${\sf Q}(\K)={\sf SPP_U}(\K)$ \cite[Theorems II.9.5
and V.2.25]{BS81}. Thus the inclusions $\sigma{\sf V}(\K)\subseteq{\sf VB}(\K)$ and $\sigma{\sf Q}(\K)\subseteq{\sf QB}(\K)$ follows from Lemma \ref{lem:: basic properties of B and O}. By Lemma \ref{lem:: preservation of quasivarieties}, $\sigma{\sf V}(\K)$ is a variety and $\sigma{\sf Q}(\K)$ is a quasivariety both containing ${\sf B}(\K)$. Thus the opposite inclusion follows. 
\end{proof}

The following lemma, in the case when $\U$ is the class of all Heyting algebra, for was proved in \cite[Theorem 3.16]{MT46}

\begin{lemma}\label{lem:: free Heyting algebras}
Let $\U$ be a universal class of Heyting algebra, $\Fr$ be a free algebra for $\U$ over $V$ and $\Fr_\sigma$ be a free algebra for $\sigma(\U)$ also over $V$. Then ${\sf B}(\Fr)$ embeds into $\Fr_\sigma$.
\end{lemma}

\begin{proof}
By Fact \ref{fact:: free algebras} point (2), $\Fr$ is free for ${\sf V}(\U)$ over $V$ and $\Fr_\sigma$ is free for ${\sf V}\sigma(\U)$ over $V$. Hence, by Lemma \ref{lem:: sigmaV=VB sigmaQ=QB}, $\Fr_\sigma$ is free for $\sigma{\sf V}(\U)$ over $V$. Thus we may assume that $\U$ is a variety. Under this assumption,  $\Fr\in\U$ and $\Fr_\sigma\in\sigma(\U)$.

 By Lemma \ref{lem::easy part of B-E thm}, 
 ${\sf O}(\Fr_\sigma)\in \U$. Thus 
 the universal mapping property  yields that there exists a homomorphism $f\colon\Fr\to{\sf O}(\Fr_\sigma)$ such that $f(v)=\mBox v$ for every $v\in V$. By Proposition \ref {prop:: McKinsey-Tarski thm} Point (3), there exists a homomorphism $f'\colon{\sf B}(\Fr)\to \Fr_\sigma$ extending $f$.

 Further,  ${\sf B}(\Fr)\in \sigma(\U)$. Thus the universal mapping property yields the existence of a homomorphism $g\colon\Fr_\sigma\to{\sf B}(\Fr)$ such that $g(v)=v$ for every $v\in V$.

We claim that $g f'\colon{\sf B}(\Fr)\to{\sf B}(\Fr)$ is an identity mapping. Indeed, since all elements from $V$ are open in ${\sf B}(\Fr)$, for every $v\in V$ we have $g f'(v)=g(\mBox v)=\mBox g(v)=\mBox v=v$. Thus, by the uniqueness of a homomorphic extension in the universal mapping property, $g f'|_F\colon \Fr\to\Fr$ is the identity mapping on $F$. And similarly, by the uniqueness in Proposition \ref {prop:: McKinsey-Tarski thm} Point (3), $g f'$ is an identity mapping. Hence $f'$ is injective.
\end{proof}

\begin{lemma}\label{lem:: free Grzegorczyk algebras}
Let $\Fr$ and $\Fr_\sigma$ be as in Lemma \ref{lem:: free Heyting algebras}, where $V$ is infinite. Then $\Fr_\sigma\in {\sf QB}(\Fr)$. 
\end{lemma}

\begin{proof}
As in the proof of Lemma \ref{lem:: free Heyting algebras}, we may assume that $\U$ and $\sigma(\U)$ are varieties. Then, by Fact \ref{fact:: free algebras} point (3) and Lemma \ref{lem:: sigmaV=VB sigmaQ=QB}, 
\[
\sigma(\U)=\sigma{\sf V}(\Fr)={\sf VB}(\Fr).
\]
Thus Fact \ref{fact:: free algebras} point (1) yields that$\Fr_\sigma\in {\sf QB}(\Fr)$.
\end{proof}

\begin{theorem}
Let $\U$ be a universal class of Heyting algebras. Then
\begin{enumerate}
\item $\U$ is strongly structurally complete if and only if $\sigma(\U)$ is strongly structurally complete;
\item $\U$ is strongly universally complete if and only if $\sigma(\U)$ is strongly universally complete.
\end{enumerate} 
\end{theorem}

\begin{proof}
We prove (1). The proof of (2) is similar. Let $\Fr$ and $\Fr_\sigma$ be as in Lemma \ref{lem:: free Heyting algebras}, where $V$ is infinite. 

Assume that $\U$ is strongly structurally structurally complete. Let us verify the condition (3) from Proposition \ref{prop:: SSC} for $\sigma(\U)$. Let $\Q'$ be a subquasivariety of ${\sf Q}\sigma(\U)$ and assume that $$\Fr_\sigma\in{\sf Q}(\sigma(\U)\cap\Q').$$ By the Blok-Esakia theorem restricted version, there is a subquasivariety $\Q$ of ${\sf Q}(\U)$ such that $\Q'=\sigma(\Q)$. By the Blok-Esakia theorem extended version and Lemma \ref{lem:: sigmaV=VB sigmaQ=QB},
\[
\Fr_\sigma\in{\sf Q}(\sigma(\U)\cap\sigma(\Q))= {\sf Q}\sigma(\U\cap\Q)
=\sigma{\sf Q}(\U\cap \Q).
\]
Thus Lemma \ref{lem:: free Heyting algebras} yields that
\[
{\sf B}(\Fr)\in\sigma{\sf Q}(\U\cap \Q).
\]
Now Lemma \ref{lem::easy part of B-E thm} and Theorem \ref{prop:: McKinsey-Tarski thm} point (1) gives that
\[
\Fr\in{\sf Q}(\U\cap \Q).
\]
Now we may apply the assumption and obtain the inclusion $\U\subseteq \Q$. Finally, by the  Blok-Esakia theorem extended version, $\sigma(\U)\subseteq \Q'$. 

For the opposite implication we proceed similarly.  
Assume that $\sigma(\U)$ is strongly structurally complete and $\Fr\in{\sf Q}(\U\cap\Q)$ for some subquasivariety $\Q$ of ${\sf Q}(\U)$. By the Blok-Esakia theorem extended version and Lemma \ref{lem:: sigmaV=VB sigmaQ=QB},
\[
{\sf B}(\Fr)\in{\sf Q}(\sigma(\U)\cap\sigma(\Q)),
\]
and by Lemma \ref{lem:: free Grzegorczyk algebras}
\[
\Fr_\sigma\in{\sf Q}(\sigma(\U)\cap\sigma(\Q)).
\]
Thus the assumption yields that $\sigma(\U)\subseteq\sigma(\Q)$, and the Blok-Esakia theorem extended version yields that $\U\subseteq\Q$.

\end{proof}

\section*{Appendix: Proof of the Blok Lemma}

We say that a Boolean algebra $\A$ is a \emph{Boolean subalgebra} of an interior algebra $\M$ if $\A$ is a subalgebra of the Boolean reduct of $\M$. Let $\M$ and $\N$ be interior algebras, $\A$ be a Boolean subalgebra of $\M$ and  $\B$ be a Boolean subalgebra of $\N$. We say that a mapping $f\colon A\to B$ is a \emph{$\mBox$-homomorphism} from $\A$ into $\B$ if it is a Boolean homomorphisms and $f(\mBox a)=\mBox f(a)$ whenever $a\in A$, and $\mBox a\in A$ (note that then $f(\mBox a)\in B$).

\begin{lemma}\label{lem:: key lemma for Blok's Lemma}
Let $\M$ be a Grzegorczyk algebra, $\B,\C$ be its finite  Boolean subalgebras and $g\in B$. Assume that
\begin{itemize}
\item  $\B$ is generated (as a Boolean algebra) by $C\cup\{g\}$,
\item  $\B$ and $\C$ share the same open elements from $\M$.
\end{itemize}
Then the identity mapping on $C$ may be extended to a $\mBox$-homomorphism from  $\B$ into $D$, where $\bf D$ is a finite Boolean subalgebra of $\M$ generated by $C$ and possibly some open elements.
\end{lemma}

\begin{proof}
Since $B$ is finite, a Boolean homomorphic extension of the identity mapping on $C$ into $M$ always exists. Indeed, by \cite[Theorem 130]{Gra11}, each such extension is given by a possible value $p$ for $f(g)$ from the nonempty interval in $\M$
\[
\Big[\bigjoin \{c\mid c\in C, c\leq g\},\bigmeet \{c\mid c\in C, g\leq c\}\Big]=[p_*,p^*]
\]
The point is to find an extension which also preserves the $\mBox$ operation.

For every $c\in C$ define
\[
p_c=\mBox(g\join c)\meet\neg c.
\]
Then
\[
p_c\leq (g\join c)\meet\neg c=g\meet \neg c\leq g,
\]
and
\begin{align*}
\mBox(g\join c)\geq \mBox(p_c\join c) &=\mBox((\mBox(g\join c)\meet\neg c)\join c)\\ &= \mBox(\mBox(g\join c)\join c) \geq \mBox\mBox(g\join c)=\mBox(g\join c).
\end{align*}
Thus
\begin{equation}\tag{P}
\mBox(p_c\join c)=\mBox(g\join c).
\end{equation}
Next, for $c\in C$ let $u=\neg g\join c$ and define
\begin{align*}
p'_c= \neg\big(\mBox(u\to \mBox u)\to \mBox u\big).
\end{align*}
Since $\M$ is an interior algebra, we have
\[
\mBox(u\to\mBox u)\to \mBox u \geq (u\to\mBox u)\to \mBox u
= u.
\]
Hence
\[
 p'_c \leq \neg(\neg g\join c)\leq g.
\]
This, in particular, gives that
\[
\mBox(\neg p'_c\join c)\geq \mBox(\neg g\join c).
\]
On the other, side we have
\begin{align*}
\mBox(\neg p'_c\join c)&=\mBox\Big(\neg\mBox(u\to \mBox u)\join \mBox u \join c\Big)\\
&\leq \mBox\Big(\neg\mBox(u\to \mBox u)\join  u\Big)\leq \mBox u=\mBox(\neg g\join c).
\end{align*}
Indeed,  the first inequality follows from the inequalities $\mBox u, c\leq u$ and the monotonicity of $\mBox$, and the second inequality follows from (Grz).
Thus
\begin{equation}\tag{P'}
\mBox(\neg p'_c\join c)=\mBox(\neg g\join c).
\end{equation}

Finally, define
\[
p=p_*\join \bigjoin_{c\in C}p_c \join \bigjoin_{c\in C}p'_c.
\]
Since $p_*,p_c,p'_c\leq g\leq p^*$ for every $c\in C$,  we have $p_*\leq p\leq p^*$. Thus the identity mapping on $C$ extends to a Boolean homomorphism $f$ from $\B$ into the Boolean reduct of $\M$, where $f(g)=p$. Let us show that $f$ also preserves the $\mBox$ operation. Let $b\in B$ be such that $\mBox b \in B$. The fact that $\B$ and $\C$ share the same open elements yields that $\mBox b\in C$, and hence $f(\mBox b)=\mBox b$. Since $\B$ is generated as a Boolean algebra by $C\cup\{g\}$, there are $c_1, c_2\in C$ such that $b=(g\join c_1)\meet (\neg g\join c_2)$ and $f(b)=(p\join c_1)\meet (\neg p\join c_2)$. Thus, since in $\M$ the $\mBox$ operations commutes with $\meet$,
\[
\mBox b=\mBox (g\join c_1)\meet \mBox (\neg g\join c_2).
\]
Further by (P) and (P'), since $\mBox$ is monotone and $p_{c_1},p'_{c_2}\leq p\leq g$,
\[
\mBox b=\mBox (p\join c_1)\meet \mBox (\neg p\join c_2).
\]
Hence
\[
f(\mBox b)=\mBox b=\mBox f(b).
\]
Finally, note that there is a finite set of open elements $X$ such that $p$ belongs to the Boolean subalgebra of $\M$ generated by $C\cup X$. For $\bf D$ we may take this algebra.

\end{proof}

\begin{lemma}\label{lem:: aux lemma for Blok's lemma}
Let $\M$ be a Grzegorczyk algebra. Let $\A$ be a finite  Boolean subalgebra of $\M$. Then there exists a $\mBox$-homomorphism  $f\colon\A\to\sf BO(\M)$ such that $f(a)=a$ for every open element $a$ from $A$.
\end{lemma}

\begin{proof}
Assume that $\A$ is generated (as a Boolean algebra) by a set $\{g_1,\ldots, g_n\}$ of non-open (in $\M$) elements and some open elements. Let us define a sequence of Boolean subalgebras $\A_i$ of $\M$ and a sequence of $\mBox$-homomorphisms $f_i$ from $\A_{i-1}$ into $\A_i$ such that
\begin{itemize}
\item $\A_i$ is a finite Boolean subalgebra of $\M$ generated by $\{g_1,\ldots, g_{n-i}\}$ and some open elements,
\item $f_i(a)=a$ for every open element $a$ from $A_{i-1}$.
\end{itemize}
In particular, $\A_n$ is a Boolean subalgebra of $\sf BO(\M)$.
Once these sequences are defined, we may define $f$ as the composition
\[
f=f_n \cdots  f_1.
\]
Let $\A_0=\A$. We proceed recursively, so let us assume that $\A_{i-1}$ is already defined.
By Lemma \ref{lem:: key lemma for Blok's Lemma}, there is a $\mBox$-homomorphisms $f_i\colon \A_{i-1}\to\A_i$, where $\A_i$ is a Boolean subalgebra of $\M$ generated by $\{g_1,\ldots, g_{n-i}\}$ and some open elements ($\A_{i-1}$ is $\B$, the Boolean subalgebra of $\A_{i-1}$ generated by $\{g_1,\ldots, g_{n-i}\}\cup O_{i-1}$, where $O_{i-1}$ is the set of open elements from $A_{i-1}$, is $\C$, and  $\A_i$ is $\bf D$ in Lemma \ref{lem:: key lemma for Blok's Lemma}).

At the end, note that $f_i(a)=a$ for every $i$ and every open element $a$ from $\A_{i-1}$. This yields that $f(a)=a$ for every open element $a$ from $A$.
\end{proof}

Now we are in position to provide the last step of the proof of the Blok Lemma.

Let $\M$ be a Grzegorczyk algebra and $\N$ be an extension of ${\sf BO}(\M)$ such that $\sf O(\M)=O(\N)$ (in particular, $\N$ might be $\M$). Let $B$ be the carrier of $\sf O(\M)$. Let $\N_B$ be the expansion of $\N$ obtained by considering every element $a\in B$  a new constant $c_a$.
Then every homomorphism $f\colon \M_B\to\N_B$ is injective.
It follows that we only need to verify the existence of any homomorphism from $\M_B$ into some elementary extension of ${\sf BO}(\M)_B$. We do this with the aid of basic model theory \cite{CK90,Hod93}.

We show that the set ${\Thel}({\sf BO}(\M)_B)\cup{\diagplus}(\M)$ is satisfiable.
Here ${\Thel}({\sf BO}(\M)_B)$ is the elementary theory of  ${\sf BO}(\M)_B$,i.e., the set of first order sentences which are valid in ${\sf BO}(\M)_B$. Note that every model of ${\Thel}({\sf BO}(\M)_B)$ is an elementary extension of ${\sf BO}(\M)_B$. Further,  ${\diagplus}(\M)$ is the positive diagram of $\M$. With every element $a\in M$ we associate a symbol of a constant $c_a$. Then ${\diagplus}(\M)$ consists of all equations of the form $c_{a}\meet c_{b}\=c_{d}$, where $a,b,d$ are such that $a\meet b=d$ in $\M$, $\neg c_{a}\= c_{}$, where $a,b$ are such that $\neg a=b$ in $\M$, and  $\mBox c_{a}\=c_{b}$, where $a,b$ are such that $\mBox a=b$ in $\M$. Here some caution is needed: In ${\diagplus}(\M)$ all symbols of constants from  ${\Thel}({\sf BO}(\M)_B)$ appear and they  correspond to open elements in $B$. There remaining symbols of constants appearing in ${\diagplus}(\M)$ are those corresponding to elements in $M-B$.

By compactness theorem, it is enough to show that every set of the form ${\Thel}({\sf BO}(\M)_B)\cup\Sigma$, where $\Sigma$ is a finite subset of ${\diagplus}(\M)$, is satisfiable. Let $\A$ be a Boolean subalgebra of $\M$ generated by the set of elements corresponding to symbols of constants appearing in $\Sigma$.
By Lemma \ref{lem:: aux lemma for Blok's lemma}, there exists a $\mBox$-homomorphism $f\colon\A\to {\sf BO}(\M)$ which fixes all elements from $A\cap B$.  Then  ${\Thel}({\sf BO}(\M)_B)\cup\Sigma$ holds in the expansion of ${\sf BO}(\M)_B$ in which every symbol of a constant occurring in $\Sigma$ and  corresponding to an  element $a$ in $A-B$ is interpreted as $f(a)$. 

\end{document}